\documentclass[ a4paper]{scrartcl}
\usepackage{amsmath, amsthm, amssymb, amsfonts}
\usepackage{graphicx}
\usepackage{xcolor}
\usepackage{booktabs}

\title{Synchronizing Times for $k$-sets in Automata}
\author{Natalie C. Behague\footnote{Department of Mathematics and Statistics, University of Victoria, Victoria, BC, Canada \quad \texttt{nbehague@uvic.ca} }, \quad Robert Johnson\footnote{School of Mathematical Sciences, Queen Mary University of London, London, UK \texttt{r.johnson@qmul.ac.uk}}}

\theoremstyle{plain}

\newtheorem{theorem}{Theorem}

\newtheorem{lemma}[theorem]{Lemma}
\newtheorem{corollary}[theorem]{Corollary}

\newtheorem{conjecture}{Conjecture}
\newtheorem{question}[conjecture]{Question}

\theoremstyle{remark}

\newtheorem{subclaim}{Claim}[theorem]

\theoremstyle{definition}

\newenvironment{proofofclaim}
{\proof[Proof of Claim]}
{\endproof}

\newenvironment{conjprime}[1]
{\renewcommand{\theconjecture}{\ref{#1}$'$}
	\addtocounter{conjecture}{-1}
	\begin{conjecture}}
	{\end{conjecture}}

\DeclareMathOperator{\Ima}{Im}
\DeclareMathOperator{\rdvs}{rdv}
\DeclareMathOperator{\Rdvs}{RDV}
\DeclareMathOperator{\vc}{\tau*}

\newcommand{\cernys}{\v{C}ern\'{y}'s }
\newcommand{\cerny}{\v{C}ern\'{y} }

\newcommand{\floor}[1]{\left\lfloor #1 \right\rfloor}
\newcommand{\ceil}[1]{\left\lceil #1 \right\rceil}

\begin{document}
\maketitle

\begin{abstract}
An automaton is synchronizing if there is a word whose action maps all states onto the same state.  \cernys conjecture on the length of the shortest such words is one of the most famous open problems in automata theory. We consider the closely related question of determining the minimum length of a word that maps some $k$ states onto a single state. 

For synchronizing automata, we find a simple argument for general $k$ almost halving the upper bound on the minimum length of a word sending $k$ states to a single state. We further improve the upper bound on the minimum length of a word sending $4$ states to a singleton from $0.5n^2$ to $\approx 0.459n^2$, and the minimum length sending $5$ states to a singleton from $n^2$ to $\approx 0.798n^2$. In contrast to previous work on triples, our methods are combinatorial. Indeed, we exhibit a fundamental obstacle which suggests that the previously used linear algebraic approach cannot extend to sets of more than 3 states.

In the case of non-synchronizing automata, we give an example to show that the minimum length of a word that sends some $k$ states to a single state can be as large as $\Theta\left(n^{k-1}\right)$. 
\end{abstract}

\section{Introduction}

A (deterministic, finite) automaton  $\Omega$ consists of a finite set of \emph{states} (usually labelled $[n] = \{1,2,\ldots,n\}$) and a finite set of \emph{mappings}, which are functions from the set of states to itself.  

We shall be interested in the results of applying a sequence of mappings to the set of states. We call such a sequence of mappings a \emph{word} of the automaton. The words of the automaton form a monoid, generated by the mappings, which acts on the set of states.

We say that a word $w$ of the automaton is a \emph{reset word} if it sends every state to the same state; that is if $w(i) = w(j)$ for all $i,j$. We call an automaton \emph{synchronizing} if it has a reset word. The most famous and long-standing open problem on synchronizing automata is \cernys conjecture.

\begin{conjecture}[\cernys Conjecture \cite{Cer}]\label{conj:cerny}
Suppose an automaton on $n$ states is synchronizing. Then the automaton has a reset word of length at most $(n-1)^2$.
\end{conjecture}

This conjecture comes from a particular family of automata, which we shall refer to as the \cerny automata. For each $n \ge 2$, we define an automaton with states $\{1,2,\ldots, n\}$ and  two mappings $f$ and $g$, defined as follows:
\[
	f(i) = i+1 \pmod{n}
	\qquad \qquad
	g(i) = 
	  \begin{cases} 
	   2 & \text{if } i = 1 \\
	   i & \text{otherwise}
	  \end{cases}
\]
Figure \ref{fig:Cerny} shows the \cerny automaton for $n=4$, which has shortest reset word $gfffgfffg$.
It is not too hard to check that the shortest reset word for the \cerny automaton on $n$ states has length $(n-1)^2$. Thus if \cernys conjecture was true the bound would be best possible. 

\begin{figure}[h]
	\centering
	\includegraphics[scale=.9]{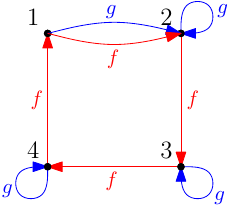}
	\caption{The \cerny automaton for $n=4$.}
	\label{fig:Cerny}
\end{figure}

\cernys conjecture has been shown to hold for certain classes of automata, 
including orientable automata \cite{Epp90}, automata where one mapping is a cyclic permutation of the states \cite{Dub98},  and automata where the underlying digraph is Eulerian \cite{Kar03}. For a survey of these and other results see \cite{Vol08} or chapter 15 of the recent Handbook of Automata Theory \cite{KV21}. It remains open to prove the conjecture for all automata. 

One can easily obtain a naive upper bound on the length of a shortest reset word by observing that
 for any pair of states there is some word sending them to a single state, and the shortest such word will never pass through the same pair of states twice. Thus the shortest word sending a given pair of states to a single state is of length at most $\binom{n}{2}$. Applying this repeatedly gives a reset word of length at most $(n-1)\binom{n}{2}$.

An improved upper bound for the length of a minimal reset word comes from a result due to Frankl and Pin \cite{Fra82} \cite {Pin83}. Rather than only considering the shortest word sending a given pair to a singleton, this instead bounds the length of the shortest word sending a given $k$-set to a $(k-1)$-set.
\begin{theorem}[Frankl--Pin]\label{thm:Frankl}
Consider a synchronizing automaton with state set $\Omega$ of size $n$.
Let $S\subseteq \Omega$ be a set of size $k$ where $k \ge 2$. There exists a word $w$ of length at most $\binom{n- k + 2}{2}$ such that $|w(S)| < k$.
\end{theorem}

Applying Theorem \ref{thm:Frankl} repeatedly, we get 
\begin{corollary}[Frankl--Pin]
An $n$-state synchronizing automaton has a reset word of length  $\le \sum_{i=2}^n \binom{n- i + 2}{2} = \frac{n^3 - n}{6}$.
\end{corollary}

This was the best known upper bound until relatively recently. Slight improvements to the constant factor have now been found:  Szyku{\l}a \cite{Szy18} obtained an upper bound of  $\approx \frac{114}{685}n^3 + O(n^2)$ and Shitov  \cite{Shi19} refined this method to obtain an upper bound of $\approx 0.1654 n^3 + o\left(n^3\right)$.

Let $\Omega$ be an automaton on $[n]$. The \emph{power automaton} $\mathcal{P}(\Omega)$, has as states the non-empty subsets of $[n]$, and for each mapping $f$ of $\Omega$ a corresponding mapping in $\mathcal{P}(\Omega)$ taking $S$ to $f(S)$ for each set of states $S$.
Figure \ref{fig:power automaton} shows the power automaton for the \cerny automaton in figure \ref{fig:Cerny}. We will think of the power automaton as a directed graph with labelled edges, and we say that the subsets of size $k$ form the \emph{$k$th layer} of the power automaton, written $L_k$.

\begin{figure}[h]
	\centering
	\includegraphics[scale=.75]{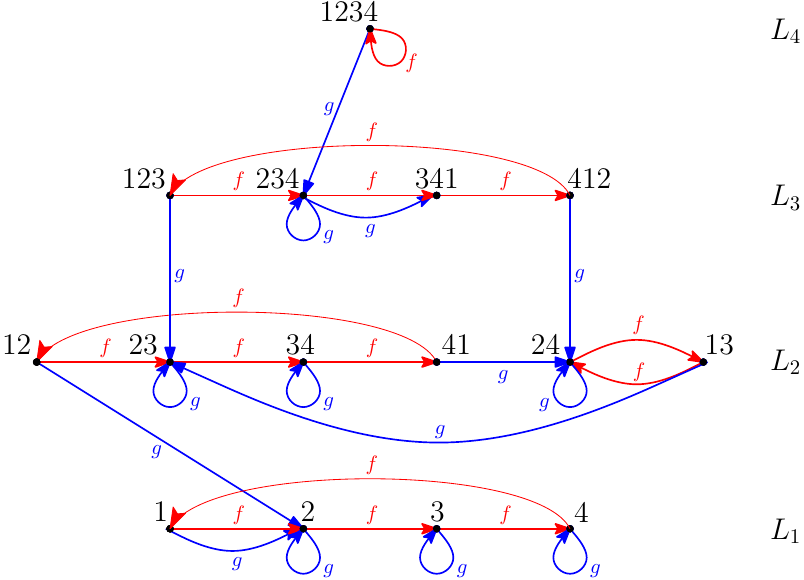}

	\caption{The power automaton for the \cerny  automaton on 4 states.}
	\label{fig:power automaton}
\end{figure}

Now \cernys conjecture can be restated in terms of the power automaton:
\begin{conjprime}{conj:cerny}[\cernys conjecture]
\label{conj:cerny2}
Let $\Omega$ be an automaton on $[n]$. 
If the power automaton $\mathcal{P}(\Omega)$ contains a path from $[n]$ to a state in $L_1$,  then there exists such a path of length at most $(n-1)^2$. 
\end{conjprime}

This formulation of \cernys conjecture suggests the more general question of determining the length of the shortest path taking a $k$-set to a singleton. This was introduced by Gonze and Jungers \cite{GJ16} as the $k$-set rendezvous time and will be our focus in the first half of this paper.

\begin{question}\label{q:min sync weight}
What is the minimum $m$ such for any synchronizing automaton $\Omega$ on $[n]$ there is a path in the power automaton $\mathcal{P}(\Omega)$  from $L_k$ to $L_1$ of length at most $m$? Denote this minimum by $\rdvs(k,n)$.
\end{question}

\begin{question}\label{q:max sync weight}
What is the minimum $m$ such for any synchronizing automaton $\Omega$ on $[n]$ and for any $k$-set $S$ there is a path in the power automaton $\mathcal{P}(\Omega)$  from $S$ to $L_1$ of length at most $m$? Denote this minimum by $\Rdvs(k,n)$.
\end{question}

Given an automaton $\Omega$ and a set of states $S$, we call $S$ \emph{synchronizable} if there exists a path from $S$ to a singleton in the power automaton $\mathcal{P}(\Omega)$. Let the weight $t(S)$ of a set $S$ be the shortest path from $S$ to a singleton if $S$ is synchronizable and $\infty$ otherwise.
We define
\begin{align*}
m(k,\Omega) &= \min\{ t(S) : S \in L_k\}  \\
M(k,\Omega) &=\max\{ t(S) : S \in L_k, S \text{ synchronizable}\}.
\end{align*}
Then $\rdvs(k,n)$ is the maximum of $m(k,\Omega)$ taken over all synchronizing automata and $\Rdvs(k,n)$ is the maximum of $M(k, \Omega)$ again taken over all synchronizing automata. It is clear that $\rdvs(k,n) \le \Rdvs(k,n)$ and $\rdvs(k,n) \le 1 + \Rdvs(k-1,n)$.

Answering either of these questions in the case $k=n$ is equivalent to answering \cernys conjecture.  
Note that finding a lower bound on $\rdvs(k,n)$ or $\Rdvs(k,n)$ requires a construction of a suitable automaton with all $k$-sets having large weight or one $k$-set having large weight respectively; while finding an upper bound on $\rdvs(k,n)$ or $\Rdvs(k,n)$ requires an argument about all synchronizing automata.

It is easy to see that $\rdvs(2,n) = 1$. We have $\Rdvs(2,n) \le \binom{n}{2}$ since at worst some pair must travel through every other pair before reaching to a singleton. In fact, $\Rdvs(2,n) = \binom{n}{2}$, where the example of a pair of weight $\binom{n}{2}$ is the pair $(2, \left\lfloor \frac{n}{2} \right\rfloor + 2)$ in the \cerny automaton on $[n]$.

The \cerny automaton also gives lower bounds for general $k$.
We have that the minimum weight $k$-set is $\{1,2,\ldots,k\}$ with weight $ (k-2)n + 1$ and so   $\rdvs(k,n) \ge (k-2)n + 1$.
Gonze and Jungers \cite{GJ16} give a construction showing that $\rdvs(3,n) \ge n+3$ for odd $n \ge 9$, which means in particular that the \cerny automaton is not extremal for the triple rendezvous time for such $n$. 
 
Using the fact that (after the first time) it takes $n$ moves to get two states one step closer to each other on the cycle, a $k$-set with states equally spaced around the circle has weight $ \ge \left\lfloor\frac{(k-1)n}{k} - 1\right\rfloor n$ and so $\Rdvs(k,n) \ge  \frac{k-1}{k}n^2 - 2n$. 
More precisely, Cardoso \cite{Car14} showed that the \cerny automaton on $n$ states contains a $k$-set with weight at least  $(n - 1)^2 - \ceil{\frac{n-k}{k}}\left(2n - k \ceil{\frac{n}{k}}-1 \right)$, and thus 
$$\Rdvs(k,n) \ge (n - 1)^2 - \ceil{\frac{n-k}{k}}\left(2n - k \ceil{\frac{n}{k}}-1 \right).$$ Cardoso conjectured that this is the true value of $\Rdvs(k,n)$.

For upper bounds on $\rdvs(k,n)$ we can apply Theorem \ref{thm:Frankl}, which gives $\rdvs(k,n) \le 1 + \sum_{i=2}^{k-1} \binom{n-i+2}{2} = \frac{k-2}{2} n^2 + O(n) $ and 
$\Rdvs(k,n) \le \sum_{i=2}^k \binom{n-i+2}{2} = \frac{k-1}{2}n^2+ O(n)$.

In Section \ref{sec:sync} we give a simple argument that $\rdvs(k,n) \le \floor{\frac{k-1}{2}}\frac{n^2}{2}$, thus almost halving the upper bound given by Frankl-Pin (indeed, for even $k$ the coefficient of $n^2$ is exactly halved).

We are also interested in upper bounds for specific values of $k$. Gonze and Jungers \cite{GJ16} proved that the triple rendezvous time $\rdvs(3,n)$ is bounded above by $\approx 0.1545 n^2 +O(n)$, using an approach based on linear programming. 
Our main results on these questions are improved upper bounds for the $4$-set and $5$-set rendezvous times $\rdvs(4,n)$ and $\rdvs(5,n)$. In particular, we prove that $\rdvs(4,n) \lessapprox 0.4589n^2 + O(n)$ and $\rdvs(5,n) \lessapprox 0.7975n^2 + O(n)$. To do so, we first prove an upper bound on $\rdvs(3,n)$ that is weaker than the bound proved in \cite{GJ16} but which uses an approach that is both purely combinatorial and, crucially, generalisable to larger $k$. The proofs of these results are given in Section \ref{sec:sync}.

In subsection \ref{sec:gonze} we include a brief discussion of the triple rendezvous bound given in \cite{GJ16}, including an explanation of a link between the linear program used and the more familiar vertex cover number, and why this suggests their approach may be difficult to generalise to $k$ larger than $3$. 

\begin{table}[h!]
\centering
	\begin{tabular}[h]{ @{}lll@{} } \toprule
		& lower bound & upper bound  \\ \midrule \addlinespace
		$\rdvs(3,n)$
			& $n + 3$
			&$\approx 0.1545n^2 + O(n)$		
		\\ \addlinespace
		$\rdvs(4,n)$
			& $2n+1$
			&\textcolor{red}{$\approx 0.4589n^2 + O(n)$}		
		\\ \addlinespace
		$\rdvs(5,n)$
			& $3n+1$
			&\textcolor{red}{$\approx 0.7975n^2 + O(n)$}	
		\\ \addlinespace
		$\rdvs(k,n)$
			& $(k-2)n + 1$
			&\textcolor{red}{$ \frac{1}{2}\left\lfloor \frac{k-1}{2}\right\rfloor n^2 + O(n)$}
		\\ \addlinespace
		$\Rdvs(k,n)$
			& $\frac{k-1}{k} n^2 + O(n) $
			& $\frac{k-1}{2}n^2+ O(n)$
			\\ \addlinespace \bottomrule
	\end{tabular}
\caption{Upper and lower bounds on $\rdvs(k,n)$ and $\Rdvs(k,n)$.}
\label{table:rdvs}
\end{table}

Table \ref{table:rdvs} summarises what is known about $\rdvs(k,n)$ and $\Rdvs(k,n)$, with the new results highlighted in red.

We can also ask similar questions over all automata, not just synchronizing automata, and this will be our focus in the second half of the paper.
\begin{question}\label{q:min weight}

What is the minimum $m$ such that for any automaton $\Omega$ on $[n]$,
if the power automaton $\mathcal{P}(\Omega)$ contains a path from $L_k$ to $L_1$,  then there exists such a path of length at most $m$? Denote this minimum by $\rdvs^*(k,n)$.
\end{question}

\begin{question}\label{q:max weight}
What is the minimum $m$ such that for any  automaton $\Omega$ on $[n]$ and for any $k$-set $S$, if there is a path from $S$ to $L_1$, then there is such a path of length at most $m$? Denote this minimum by $\Rdvs^*(k,n)$.
\end{question}

In particular, $\rdvs^*(k,n)$ is the maximum of $m(k,\Omega)$ taken over all automata $\Omega$ with at least one synchronizable $k$-set, and $\Rdvs^*(k,n)$ is the maximum of $M(k,\Omega)$ over the same collection of automata.

Again we have that answering either question in the case $k=n$ is again equivalent to \cernys conjecture. Note that $\rdvs^*(k,n) \le \Rdvs^*(k,n)$ and $\rdvs^*(k,n) \le 1 + \Rdvs^*(k-1,n)$.

A similar question on the maximum length of a word resetting some subset of states of a non-synchronizing automaton was studied by Vorel \cite{Vor16} among others. In particular, Vorel proved that there exist non-synchronizing automata on $n$ states with subsets taking $2^{\Omega(n)}$ steps to synchronize.  We ask the slightly more specific question of what happens for sets of $k$ states.

A naive upper bound on $\rdvs^*(k,n)$ is $1 + \sum_{i=2}^{k-1} \binom{n}{i}$,  since a shortest word down to a singleton will take a set through each set of size $< k$ at most once. 

A very slightly improved upper bound can be obtained by noting that an automaton is synchronizing if and only if for every pair of states $u,v$ there is a word $w$ with $w(u) = w(v)$.  If the automaton is synchronizing then we can use the Frankl--Pin bound. If not, then there is a pair $u,v$ that cannot be sent to the same state and any set containing both $u$ and $v$ is not synchronizable. The shortest path will not pass through any of these sets and so $\rdvs^*(k,n) \le  1 + \sum_{i=2}^{k-1}\left(\binom{n}{i} - \binom{n-2}{i-2} \right)  $.
In either case, for fixed $k$ we have $\rdvs^*(k,n) = O(n^{k-1})$ and by the same argument $\Rdvs^*(k,n) = O(n^{k})$.

In Section \ref{sec:constructions} we show that this is best possible --- that is, if $k$ is fixed then the answer to Question \ref{q:min weight} is $\Theta\left(n^{k-1}\right)$. Since $\Rdvs^*(k,n) \ge \rdvs^*(k+1,n) - 1$, we also get that $\Rdvs^*(k,n) = O\left(n^k\right)$. The non-synchronizing case therefore exhibits very different behaviour to the synchronizing case, which implies that any approach to \cernys conjecture using rendezvous times must use the condition that the automata is synchronizing in a critical way.

\begin{table}[h]
\centering
	\begin{tabular}[h]{ @{}lll@{} } \toprule
		& lower bound & upper bound \\ \midrule \addlinespace
		$\rdvs^*(3,n)$
			& \textcolor{red}{$ \frac{1}{8}n^2$ }
			& $ \frac{1}{2}n(n-1)$ 			
		\\ \addlinespace
		$\rdvs^*(k,n)$	
			&\textcolor{red}{ $ \frac{4}{3}\left(\frac{n}{4k}\right)^{k-1} $ }
			&$ \binom{n}{k-1} + O\left(n^{k-2}\right)$
		\\ \addlinespace
		$\Rdvs^*(k,n)$
		& \textcolor{red}{$ \frac{4}{3}\left(\frac{n}{4(k+1)}\right)^{k} -1 $ }
			& $ \binom{n}{k} + O\left(n^{k-1}\right)$
			\\ \addlinespace \bottomrule
	\end{tabular}
\caption{Upper and lower bounds on $\rdvs^*(k,n)$ and $\Rdvs^*(k,n)$.}
\label{table:rdvs*}
\end{table}

Table \ref{table:rdvs*} summarises what is known about $\rdvs^*(k,n)$ and $\Rdvs^*(k,n)$. The new contributions are highlighted in red.

For the avoidance of confusion please note that throughout this paper we use the convention that when applying a word $f_nf_{n-1}\ldots f_2f_1$ to a state $v$, we first apply the mapping $f_1$, then the mapping $f_2$ and so on, as with composition of functions.

\section{Upper Bounds on the Rendezvous Time}
\label{sec:sync}
Frankl--Pin gives trivially that for $2 \le k \le n$, the $k$-set rendezvous time $\rdvs(k,n)$ is at most $1 +  \sum_{i=2}^{k-1}\binom{n- i+2}{2}.$ The following simple adaptation of Frankl--Pin's result improves on this bound for $k\ge 4$.

\begin{theorem} \label{thm:weak rendezvous bound}
For all $n$ and all $2\le k\le n$ the $k$-set rendezvous time $\rdvs(k,n)$ is  at most
$$ \sum_{i=1}^{\left\lfloor\frac{k}{2}\right\rfloor}\binom{i+1}{2} + \sum_{i=1}^{\left\lceil\frac{k}{2}\right\rceil - 1}\binom{n-i+1}{2}.$$
In particular, for fixed $k$ and $n$ sufficiently large given $k$ we have that $$\rdvs(k,n) < \left\lfloor \frac{k-1}{2}\right\rfloor \frac{n^2}{2} .$$
\end{theorem}
\begin{proof}
By Frankl--Pin, there exists a word $w$ that takes $[n]$ to a set $S$ of size $n-\left\lfloor\frac{k}{2}\right\rfloor$ of length at most $$\sum_{i=n - \left\lfloor\frac{k}{2}\right\rfloor + 1}^{n}\binom{n- i+2}{2} = \sum_{i=1}^{\left\lfloor\frac{k}{2}\right\rfloor}\binom{i+1}{2}. $$

By the pigeonhole principle, there are at least $n - 2\left\lfloor\frac{k}{2}\right\rfloor$ states in $S$ with exactly one state in their pre-image under $w$. 
Take $T$ to be $ n - k$ such states and let $R = S\setminus T$. We have $|R| = \left\lceil\frac{k}{2}\right\rceil $ and $|w^{-1}(R)| = n - |w^{-1}(T)| = n - |T| = k$.

By Frankl--Pin again, we can find a word $w'$ that takes $R$ to a singleton of length at most
$$ \sum_{i=2}^{\left\lceil\frac{k}{2}\right\rceil}\binom{n-i+2}{2} =  \sum_{i=1}^{\left\lceil\frac{k}{2}\right\rceil - 1}\binom{n-i+1}{2}.$$
Concatenating $w'w$ gives the required word.
\end{proof}

We can also obtain improved bounds on the rendezvous time for $k=4$ and $5$. We shall first consider the case $k=3$. The triple rendezvous time $\rdvs(3,n)$ was studied in particular by Gonze and Jungers \cite{GJ16}, who proved that $\rdvs(3,n) \le \frac{\sqrt{5}-1}{8}n^2 + O(n) \approx 0.1545n^2 + O(n)$. Theorem \ref{thm:k=3} gives a weaker upper bound 
(note that $\frac{3-\sqrt{5}}{4} \approxeq 0.1910$) but it serves to illustrate the combinatorial approach that we will apply to the more complicated cases of $k=4$ and $k=5$.

\begin{theorem} \label{thm:k=3}
	For all $n\ge3$, we have $\rdvs(3,n) \le \frac{3-\sqrt{5}}{4}n^2 + \frac{3}{2}n.$
\end{theorem}

For the sake of clarity, we will first prove the bound for strongly connected automaton, from which Theorem \ref{thm:k=3} will follow as a corollary. An automaton is \emph{strongly connected} if for any ordered pair of states $u, v$ there is a word $w$ such that $w(u) = v$. 

We will need two further definitions. The \emph{image} $\Ima{w}$ of a word $w$ is the set of states that are in the image of $w$ thought of as a function. In particular, $\Ima{w} = w([n])$. The \emph{rank} of a word is the number of states in its image.

\begin{theorem} \label{thm:k=3_strongly_connected}
	Let $\Omega$ be a strongly connected synchronizing automaton on $n \ge 3$ states. There exists some set $S$ of three states and some word $w$ of length $\le \frac{3-\sqrt{5}}{4}n^2 + \frac{3}{2}n$ such that $w(S)$ is a single state.
\end{theorem}

\begin{proof}
	First, note that if $n=3$ then the triple rendezvous time is at most $4$ and the result is trivially true.
	Thus we may assume that $n \ge 4$. 
	
	Let $r$ be the minimum rank over all words of length at most $n$. Note that by Frankl--Pin there is a word of length $4 = \binom{2}{2} + \binom{3}{2}$ that takes $[n]$ to a set of size $n-2$. Since $n \ge 4$ we thus have that $r \le n-2$.
	
	Let $w$ be a word of length $\le n$ of minimal rank $r$. If $r< \frac{n}{2}$ then by the pigeonhole principle there must be some triple sent to a singleton by $w$ and so we are done. We may therefore assume that $r \ge \frac{n}{2} \ge 2$.

	\begin{subclaim} \label{claim1}
		There exists a word of length $\le n + \binom{r+2}{2}$ that takes some  triple to a singleton.
	\end{subclaim}
	\begin{proofofclaim}
		Let $w$ be a word of length $\le n$ of minimal rank $r$. If there is some triple which $w$ sends to a singleton then we are done, so we can assume that $w$ sends at most two states to the same state.
		
		Let ${S = \{x: \exists y \ne z \text{ with } w(y) = w(z) = x\}}$  be the set of states with pre-images of size 2 under $w$. Let $T = \Ima{w}- S$ be the set of all states with a singleton pre-image under $w$.
		We have that $|S| + |T| = r$ and $2|S| + |T|= |w^{-1}([n])| = n$, from which we obtain $|S| = n-r$.
		
		By Frankl--Pin there exists a word $w'$ of length $\le \binom{n-(n-r)+2}{2} = \binom{r+2}{2}$ such that $|w'(S)|<|S|$.
		In particular, there exist $x \ne y$ in $S$ with $w'(x) = w'(y) = z$. Take states $u,v$ with $w(u)=w(v) = x$ and $s,t$ with $w(s) =w(t)=y$. 
		The word $w' w$ has length at most $n + \binom{r+2}{2}$ and $w'w(\{u,v,s,t\}) = w'\{x,y\}) = \{z\}$ so in this case $w'w$ sends some 4-set to a single state.
	\end{proofofclaim}
	
	\begin{subclaim} \label{claim2} There exists a word of length $\le n + \frac{(n-r)n}{2}$ that takes some triple to a singleton.
	\end{subclaim}
	\begin{proofofclaim}
		We call an pair of states $(u,v)$ \emph{good} if there exists a word $w_{uv}$ of length $\le n$ with $ |w_{uv}^{-1}(\{u,v\})| \ge 3$. We will count the number of good pairs.
		
		We can find a state $z$ and a mapping $f$  with $|f^{-1}(z)| \ge 2$. Since $\Omega$ is strongly connected, for each state $v$ there is some word $w_v$ of length $\le n-1$ with $w_v(z) = v$. In particular, $(w_vf)^{-1}(v) \supseteq f^{-1}(z)$ where $w_vf$ is a word of length $\le m$. For all $a \in (\Ima{(w_vf)} \setminus \{v\} )$ the pair $(v,a)$ is good.
		
		Since $w_vf$ is a word of length $\le n$, the rank of $w_vf$ is $\ge r$ and so $|\left(\Ima{(w_vf)} \setminus \{v\}\right)| \ge r - 1 > 0$.  Every state is in at least $r - 1 $ good pairs and so the number of good pairs is at least $\frac{(r-1) n}{2} > 0$. 
		
		The number of pairs in $E$ that are not good is at most $\binom{n}{2} - \frac{(r-1) n}{2} = \frac{(n-r)n}{2}$. We conclude that there is some word $w$  of length at most $\frac{(n-r)n}{2}$  that sends some good pair $(u,v)$ to a singleton $x$, where the worst case scenario is having to pass through every not good pair first.
		
		By definition of good, we can find a word $w_{uv}$ of length $\le n$ with $|w_{uv}^{-1}(\{u,v\})| \ge 3$. 
		Then $w w_{uv}$ is a word of length at most $n + \frac{(n-r)n}{2}$ where $(w w_{uv})^{-1} (x) \supseteq w_{uv}^{-1}(\{u,v\})$ has size at least $3$. In particular, the claim holds.
	\end{proofofclaim}

	Combining the results of these two claims, we have that the triple rendezvous time for strongly connected automata is at most $\min\left\{ n + \binom{r+2}{2}, n + \frac{(n-r)n}{2} \right\}$. The former is increasing in $r$ while the latter is decreasing in $r$ and so to find the maximum we look for the $r$ where they are equal. This is when $(r+2)(r+1) = (n-r)n$, which occurs when $r = \frac{-n-3 + \sqrt{5n^2 + 6n +1}}{2}$  (subject to $r\ge 0$). 
	
	Substituting this in gives that the triple rendezvous time for strongly connected automata is at most 
	\begin{align*}
	n + \frac{\left(3n+3 - \sqrt{5n^2 + 6n + 1}\right) n}{4}
	& \le n + \frac{\left(3n+3 - \sqrt{5}\left(n + \frac{1}{\sqrt{5}}\right) \right) n}{4}
	\\
	&= n + \frac{\left((3-\sqrt{5})n + 2 \right) n}{4} \\
	&= \frac{3-\sqrt{5}}{4}n^2 + \frac{3}{2}n.
	\end{align*}
	
\end{proof}

We can use the strongly connected case to prove Theorem \ref{thm:k=3}.

\begin{proof}[Proof of Theorem \ref{thm:k=3}]
Let $C$ be the sink component of the automaton, that is, the minimal non-empty set of states such that $f(C)\subseteq C$ for every $f$. 
In particular, if $w$ is any synchronizing word and $x$ is the state with $w([n]) = x$ then 
${C = \{y : \exists w' \text{ with } w'(x) = y\}}$. 
Let $m = |C|$ and let $\Omega'$ be the automaton restricted to $C$ (with state set $C$ and mappings that are the mappings of $\Omega$ restricted to $C$).

Note that $\Omega'$ is strongly connected, and so if $m \ge 3$ we can apply Theorem \ref{thm:k=3_strongly_connected} to $\Omega'$. We obtain a word sending a triple to a singleton of length $\le \frac{3-\sqrt{5}}{4}m^2 + \frac{3}{2}m \le \frac{3-\sqrt{5}}{4}n^2 + \frac{3}{2}n$.

Suppose that $m \le 2$. If $m=2$, write $C = \{x,y\}$. Since the automaton is synchronising, there must be some mapping $g$ with $g(x) = g(y)$. In addition, in $\Omega$ there must be some state $z \ne x,y$ and some mapping $f$ such that $f(z) \in C$. Note that $f(C) \subseteq C$, and so the word $gf$ of length 2 takes the triple $\{x,y,z\}$ to a singleton.

Finally, suppose $m=1$ and let $C = \{x\}$. In $\Omega$, there must be some state $y \ne x$ and some mapping $f$ such that $f(y)=x$. Let $r$ be the rank of $f$.
If $r < \frac{n}{2}$, then by the pigeonhole principle there must be some triple sent to a singleton by $f$ and we are done. Suppose $r \ge \frac{n}{2}$. 
For every state of $\Omega$ there is some word taking that state to $x$. Thus we can find  some state $z$ in $\Ima{f} \setminus \{x\}$ and some word $w$ of length at most $n-r+1$ with $w(z) = x$, where the worst case scenario is having to pass through every state in $[n] \setminus \Ima{f} $ before reaching $x$. 
Now $wf$ is a word of length at most $n-r + 2 \le \frac{n}{2} + 2$ that takes the triple $\{x,y,z\}$ to the singleton $x$.
\end{proof}

In a similar way we can improve the upper bounds on the $4$-set and $5$-set rendezvous times.

\begin{theorem} \label{thm:k=4}
For all $n \ge 4$, we have 
$$\rdvs(4,n) \le \rdvs(3,n) +   (2 - \sqrt{3})n^2 + 2n - 1 \le \left(\frac{11-\sqrt{5} - 4\sqrt{3}}{4}\right) n^2 + \frac{7}{2}n -1.$$
\end{theorem}
Note that $ \frac{11-\sqrt{5} - 4\sqrt{3}}{4} \approxeq  0.4589$, so this is again an improvement on the $4+\binom{n}{2}$ given by Theorem \ref{thm:weak rendezvous bound}.

\begin{theorem} \label{thm:k=5}
For all $n \ge 5$, we have 
$$\rdvs(5,n) \le \rdvs(4,n) +  \frac{4 - \sqrt{7}}{4}n^2 + \frac{3}{2}n -1 
\le \left(\frac{15-\sqrt{5} - 4\sqrt{3}-\sqrt{7}}{4}\right) n^2 + 5n -2.$$
\end{theorem}
Note that $ \frac{15-\sqrt{5} - 4\sqrt{3}-\sqrt{7}}{4} \approxeq  0.7975$, so this is again an improvement on the bound of $4+\binom{n}{2} + \binom{n-1}{2}$ given by Theorem \ref{thm:weak rendezvous bound}.

To prove both Theorems \ref{thm:k=4} and \ref{thm:k=5} we will need two lemmas. In the case $k=2$ the lemmas correspond precisely to claims \ref{claim1} and \ref{claim2} in the proof of Theorem \ref{thm:k=3} and we prove each lemma in an analogous way.
\begin{lemma} \label{lem:1}
Fix $2 \le k \le n-1$ and $l \ge 1$. Let $\Omega$ be a synchronizing automaton on $n$ states and let $r$ be the minimal rank over all words of length $\le l $ in $\Omega$. 
Suppose that $r \le \frac{n-2 \left\lceil \frac{k}{2} \right \rceil}{ \left\lfloor \frac{k}{2} \right \rfloor}$ 
 and let $s = \frac{n -  \left\lfloor \frac{k}{2} \right \rfloor r}{\left\lceil \frac{k}{2} \right \rceil} \ge 2$.
Then $$\rdvs(k+1,n) \le l + \binom{n - s + 2}{2}.$$
\end{lemma}
\begin{proof}
Let $w$ be a word of length $\le l$ of minimal rank $r$. If there is some $(k+1)$-set which $w$ sends to a singleton then we are done, so we can assume that $w$ sends at most $k$ states to the same state.

Let ${S = \{x: \left|w^{-1}(x)\right| \ge  \left\lceil \frac{k+1}{2} \right\rceil \} }$  be the set of states with at least $ \left\lceil \frac{k+1}{2} \right \rceil$ pre-images under $w$. Let $T = \Ima{w}- S$   be the set of states with $\le \left\lfloor \frac{k}{2} \right \rfloor$ pre-images under $w$.
We have that $|S| + |T| = r$. We also have that $n = |w^{-1}([n])| = |w^{-1}(S)| + |w^{-1}(T)| \le k|S| + \left\lfloor \frac{k}{2} \right \rfloor|T|$. Putting these together, we have $k|S| +  \left\lfloor \frac{k}{2} \right \rfloor (r-|S|) \le n$ and so $|S| \ge \frac{n -  \left\lfloor \frac{k}{2} \right \rfloor r}{\left\lceil \frac{k}{2} \right \rceil} \ge 2.$

By Frankl--Pin there exists a word $w'$ of length $\le \binom{n-|S|+2}{2}$ such that $|w'(S)|<|S|$.
In particular, there exist $x \ne y$ in $S$ with $w'(x) = w'(y) = z$.
We have that $\left|(w'w)^{-1}(z)\right| = |w^{-1}(x,y)| \ge 2\left\lceil \frac{k+1}{2} \right\rceil$, so $w'w$ sends some $(k+1)$-set to a single state.

The length of the word $w'w$ is 
$l + \binom{n-|S| + 2}{2} \le l +  \binom{n- s + 2}{2} $ where $s =  \frac{n -  \left\lfloor \frac{k}{2} \right \rfloor r}{\left\lceil \frac{k}{2} \right \rceil}$.
\end{proof}

\begin{lemma}\label{lem:2}
	Fix $3 \le k \le n-1$ and $l \ge 1$. Let $\Omega$ be a synchronizing automaton on $n$ states and let $r$ be the minimal rank over all words of length $\le l $ in $\Omega$. 
	Then there exists a word of length $\le \rdvs(k,n) + n - 1 + \frac{(n-r)n}{2}$ that sends some $(k+1)$-set to a singleton.
\end{lemma}

We will first prove this Lemma for strongly connected automata.
\begin{lemma} \label{lem:2_strongly_connected}
	Fix $k \le n-1$. Let $\Omega$ be a strongly connected synchronizing automaton on $n$ states and let $r$ be the minimal rank over all words of length $\le \rdvs(k,n) + n - 1$. Then there exists a word of length $\le \rdvs(k,n) + n - 1 + \frac{(n-r)n}{2}$ that sends some $(k+1)$-set to a singleton.
\end{lemma}

\begin{proof}
	Let $E = \{(u,v) : u \ne v\}$ be the set of pairs of states.
	Let $l = \rdvs(k,n)$. We call an pair $(u,v) \in E$ \emph{good} if there exists a word $w_{uv}$ of length $\le l + n - 1$ with $ |w_{uv}^{-1}(\{u,v\})| \ge k + 1$. We will count the number of good pairs.

	There is a word $w$ of length $l = \rdvs(k,n)$ that sends some $k$-set to a singleton $x$.
	
	Since $\Omega$ is strongly connected, for each state $v$ there is some word $w_v$ of length $\le n-1$ with $w_v(x) = v$. In particular, $(w_vw)^{-1}(v) \supseteq (w)^{-1}(x)$ where $w_vw$ is a word of length $\le l + n - 1$. For all $a \in (\Ima{(w_vw)} \setminus v )$ the pair $(v,a)$ is good.
	
	Since $w_vw$ is a word of length $\le l + n - 1$, the rank of $w_vw$ is $\ge r$ and so $|\left(\Ima{(w_vw)}\setminus \{v\}\right)| \ge r - 1 > 0$.  
	Every state in $C$ is in at least $r - 1$ good pairs and so the number of good pairs is at least $\ge \frac{(r-1) n}{2} > 0$. 
	
	The number of pairs in $E$ that are not good is at most $\binom{n}{2} - \frac{(r-1)n}{2} = \frac{(n-r)n}{2}$. 
	We conclude that there is some word $w$  of length at most $\frac{(n-r)n}{2}$  that sends some good pair $(u,v)$ to a singleton $x$, where the worst case scenario is having to pass through every not good pair first.
	
	By definition of good, we can find a word $w_{uv}$ of length $\le n$ with $|w_{uv}^{-1}(\{u,v\})| \ge k+1$. 
	Then $w w_{uv}$ is a word of length at most $l + n-1 + \frac{(n-r)n}{2}$ where $(w w_{uv})^{-1} (x) \supseteq w_{uv}^{-1}(\{u,v\})$ has size at least $k+1$. In particular, the claim holds.
\end{proof}

We can now prove the more general statement for non-synchronizing automata.
\begin{proof}[Proof of Lemma \ref{lem:2}]
Let $l = \rdvs(k,n)$.
Let $C$ be the sink component of the automaton and let $m = |C|$. 
Let $\Omega'$ be the automaton restricted to $C$ and let $r'$ be the minimal rank of a word of length $l + n-1$ restricted to $\Omega'$. 

Note that $\Omega'$ is strongly connected, and so if $m \ge k$ we can apply Theorem \ref{lem:2_strongly_connected} to $\Omega'$. Since $|\Ima_{\Omega'} f| = |\Ima_{\Omega} f \cap C| \ge |\Ima_{\Omega} f|- (n-m)$, we have $r' \ge r - (n-m)$ and in particular, $m-r' \le n-r$.
We obtain a word sending a $(k+1)$-set to a singleton of length $\le l + m-1 + \frac{(m-r')m}{2} \le l + n-1 + \frac{(n-r)n}{2}$ and we are done.

Suppose then that $m \le k-1$. 
We know that there is a word $w$ in $\Omega$ of length $\le l$ that sends a $k$-set to a singleton and which therefore has rank $\le n-k + 1$. This tells us that $r \le n-k+1$. In particular, we have $m \le k-1 \le n-r$.
We will show that there is a word of length $\le l + 2(n-m) + \binom{m}{2}$ that takes some $(k+1)$-set to a singleton.

There is a word $w$ of length $l = \rdvs(k,n)$ that sends some $k$-set to a singleton $x$.
We can then find a word $w_1$ of length at most $n-m$ that sends $x$ to a state $z \in C$, where at worst we have to go through every state not in $C$ before we reach $C$. 
In particular, $|(w_1w)^{-1} (z) | \ge k$ where $w_1w$ is a word of length $\le l + n - m$.

If $z$ is the only state in $\Ima( w_1w)$ 
then $w_1w$ is a synchronizing word sending 
$n \ge k+1$ states down to a singleton. 
So suppose $\Ima(w_1w) - z$ is non-empty and take some $v \in \Ima(w_1w)$, $v \ne z$. 
We can find a word $w_2$ of length $\le n-m$ that takes $v$ to a vertex $y \in C$.

If $w_2(z) = y$ then the word $w_2 w_1 w$ of length $\le l + 2(n-m)$ has $|(w_2 w_1 w)^{-1}(y)| = |(w_1w)^{-1}(z,v)| \ge k + 1$. 
Otherwise, $y, w_2(z)$ are two distinct vertices in $C$ and we can find a word $w_3$ of length at most $\binom{m}{2}$ that takes $\{y, w_2(z)\}$ to a singleton. 
Then $w_3w_2w_1w$ is a word of length $\le l + 2(n-m) + \binom{m}{2}$ that takes some $(k+1)$-set to a  singleton.

To prove that the bound as stated in the lemma holds, it suffices to show that the following quantity is positive. 
$$
\left(n - 1 + \frac{(n-r)n}{2}\right) - \left(2(n-m) + \binom{m}{2}\right) 
= \frac{ n(n-r-2) -m^2 + 5m - 2}{2}
$$
Note that since there is a word of length $\le l$ that sends a $k$-set to a singleton, we have a word of length $l$ of rank $\le n-k + 1$. This implies that $r \le n-k+1$, and so in particular $n-r-2 \ge 0$.
If $m=1$, we are already done. Otherwise, $m \ge 2$ and substituting $m \le n-r$, we obtain
\begin{align*}
\frac{ n(n-r-2) -m(m-5) - 2}{2} &\ge \frac{ (m+r)(m-2) -m^2 + 5m - 2}{2}\\
&= \frac{(m-2)r+3m -2}{2} \ge 0.
\end{align*}

\end{proof}

We use these lemmas to prove the Theorems.
\begin{proof}[Proof of Theorem \ref{thm:k=4}]
Fix $n \ge 4$. Let $l =  \rdvs(3,n) + n-1$ and let $r$ be the minimal rank of a word of length at most $l$. 

Applying the $k=3$ case of Lemmas \ref{lem:1} and \ref{lem:2} we get
$$\rdvs(4,n) \le 
	\begin{cases} 
		l + \frac{1}{2}\left(\frac{n+r}{2}+2\right)\left(\frac{n+r}{2}+1\right) &\text{if } r\le n-4\\
		 l + \frac{1}{2}(n-r)n &\text{for all } r
	\end{cases}
$$

If $r > n-4$ then $ \frac{1}{2}(n-r)n < 2n$.

If $r \le n-4$, the first bound is increasing with $r$ (for $r \ge 0$) and the second is decreasing with $r$ so the maximum is obtained where the two are equal, that is when 
$\left(\frac{n+r}{2}+2\right)\left(\frac{n+r}{2}+1\right) = (n-r)n$. Rearranging gives 
$r^2 + 6(n+1)r -(3n^2 - 6n - 8) = 0$.
Solving for $r$, we get that the maximum is obtained when 
$$ r = -3(n+1) + \sqrt{12n^2 + 12n + 1}$$

Thus we have that the maximum is
\begin{align*}
\frac{(n-r)n}{2} &= \frac{\left(4n + 3 - \sqrt{12n^2 + 12n + 1}\right)n}{2} \\
& \le \left(2n + \frac{3}{2} - \sqrt{3}\left(n + \frac{1}{\sqrt{12}}\right)\right)n \\
& = (2 - \sqrt{3})n^2 + n
\end{align*}

Putting this together with the bound on $\rdvs(3,n)$ from Theorem \ref{thm:k=3} we get the final bound 
$$\rdvs(4,n) \le \left(\frac{3-\sqrt{5}}{4} + 2-\sqrt{3}\right) n^2 + \frac{7}{2}n - 1.$$

\end{proof}

\begin{proof}[Proof of Theorem \ref{thm:k=5}]
Fix $n \ge 5$. Let $l =  \rdvs(4,n) + n-1$ and let $r$ be the minimal rank of a word of length at most $l$. 

Applying the $k=4$ case of  Lemmas \ref{lem:1} and \ref{lem:2} we get
$$\rdvs(5,n) \le 
	\begin{cases} 
		l + \frac{1}{2}\left(\frac{n+2r}{2}+2\right)\left(\frac{n+2r}{2}+1\right) &\text{if } r\le \frac{n-4}{2}\\
		 l + \frac{1}{2}(n-r)n &\text{for all } r
	\end{cases}
$$

If $r > \frac{n-4}{2}$ then $ \frac{1}{2}(n-r)n < \frac{1}{4}n^2 +n $.

If $r \le \frac{n-4}{2} $, the first bound is increasing with $r$ (for $r \ge 0$) and the second is decreasing with $r$ so the maximum is obtained where the two are equal, that is when 
$\left(\frac{n+2r}{2}+2\right)\left(\frac{n+2r}{2}+1\right) = (n-r)n$. Rearranging gives 
$r^2 + (2n+3)r -(\frac{3}{4}n^2 - \frac{3}{2}n - 2) = 0$.
Solving for $r$, we get that the maximum is obtained when 
$$ r = \frac{-(2n+3) + \sqrt{7n^2 + 6n + 1}}{2}$$

Thus we have that the maximum is
\begin{align*}
\frac{(n-r)n}{2} &= \frac{\left(4n + 3 - \sqrt{7n^2 + 6n + 1}\right)n}{4} \\
& \le \frac{\left(4n + 3 - \sqrt{7}\left(n + \frac{1}{\sqrt{7}}\right)\right)n}{4} \\
& = \frac{4 - \sqrt{7}}{4}n^2 + \frac{1}{2}n
\end{align*}
Putting this together with the bound on $\rdvs(4,n)$ from Theorem \ref{thm:k=3} we get the final bound.

\end{proof}

It is clear that we could continue applying this method in the way we have here to obtain upper bounds on $\rdvs(k,n)$ for larger $k$. However, as it stands the method does not give an improvement on the bound $\rdvs((k,n) < \left\lfloor \frac{k-1}{2}\right\rfloor \frac{n^2}{2}$ given by Theorem \ref{thm:weak rendezvous bound} for larger $k$. We remain hopeful that the method could be improved upon to give results for larger $k$. One approach might be to alter Lemma \ref{lem:1} to allow one to go directly from a result about $\rdvs(k,n)$  to a result about $\rdvs(k+c,n)$ for $c$ larger than $1$.

\subsection{A Discussion of Gonze and Jungers' Result} \label{sec:gonze}
Gonze and Jungers prove in Theorem 3.13 of \cite{GJ16} that for strongly connected synchronizing automata the triple rendezvous time $\rdvs(3,n)$ is at most $\frac{\sqrt{5}-1}{8}n^2 + O(n) \le 0.16 n^2 + O(n)$. To do so, they used a linear program first introduced in \cite{Jun12}. 

We will discuss a correspondence between this linear program  and the more well-known fractional vertex cover number. 
This connection to a better understood linear program is interesting in its own right and could potentially lead to improvements to the triple rendezvous time. Unfortunately, this correspondence also suggests that this linear programming approach cannot be straightforwardly generalised to give upper bounds on $k$-set rendezvous times for $k>3$. We explain why at the end of this section.
 
First, let us state the original linear program defined in \cite{Jun12}. Fix an automaton $\Omega$ on $n$ states.
If $t < \rdvs(3,n)$, the sets for which there is a word of length $\le t$ sending that set to a singleton will only be singletons and pairs. We let $G_t$ be the graph on $n$ states with edge-set all such pairs, and let $m(t) = e\left(G_t\right)$ be the number of such pairs. 

Let $A(t)$ be a matrix with rows indexed by $[n]$ and columns indexed by the sets of $\Omega$ that can be sent to a singleton by a word of length $\le t$. A column corresponding to the set $S$ will have a $1$ in rows indexed by elements of $S$ and a $0$ in all other rows. For example, $A(0)$ will be the $n \times n$ identity matrix. In general, $A(t)$ will have $n+m(t)$ columns for $t<\rdvs(3,n)$.

Define $Prog_{A(t)}$ to be the linear program
\begin{equation}
	\begin{aligned}[b] \label{eqn:GonzeLP}
		\min_{\textbf{p},k} k \\
		\text{s.t.~~~} &\textbf{p}A(t) \le k\textbf{e}_{n+m(t)} \\
		&\textbf{e}_n\textbf{p}^{\mathsf{T}} = 1 \\
		&\textbf{p} \ge \textbf{0} 
	\end{aligned} 
\end{equation}
where $\textbf{p}$ is a row vector of length $n$ and $\textbf{e}_i$  is the all ones row vector of length $i$.
Let $k(t)$ be the minimum value attained by $Prog_{A(t)}$ and let $P_t$ be the set of optimal solutions $\textbf{p}$ to $Prog_{A(t)}$.  

The linear program $Prog_{A(t)}$ can be thought of in terms of assigning weights to $G_t$. The vector $\textbf{p}$ assigns a weight to each vertex of $G_t$ such that the sum of all the weights is one. The condition $\textbf{p}A(t) \le k\textbf{e}_{n+m(t)}$ means that for each vertex and edge of $G_t$ the sum of the weights on the incident vertices is at most $k$. Then $k(t)$ is minimal subject to a weighting of $G_t$ existing that satisfies these conditions. 

A critical part of the proof in \cite{GJ16} is that the set of possible minimum values $k(t)$ attained by $Prog_{A(t)}$ is small. This is used together with an argument showing that as $t$ increases, either $k(t)$ decreases or the dimension of the solution space decreases and so there is a bound on how large $t$ can grow.

We can rewrite $Prog_{A(t)}$ in the form of a more well-known linear program. We will do this in two stages. 
First, define $E(t)$ to be a matrix with rows indexed by $[n]$ and columns indexed by the edges of $G_t$, where the column indexed by edge $ij$ has a $1$ in rows $i$ and $j$ and $0$s elsewhere. Now $A(t)$ is the concatenation of the $n \times n$ identity matrix $I_n$ and $E(t)$, and so the statement $\textbf{p}A(t) \le k\textbf{e}_{n+m(t)}$ is equivalent to the statements $\textbf{p} \le k\textbf{e}_n$ and $\textbf{p}E(t) \le k\textbf{e}_{m(t)}$. We rewrite $Prog_{A(t)}$ as follows: 

	\begin{align*} \label{eqn:GonzeLP2}
		\min_{\textbf{p},k} k \\
		\text{s.t.~~~} &\textbf{p}E(t) \le k\textbf{e}_{m(t)} \\
		&\textbf{e}_n\textbf{p}^{\mathsf{T}} = 1 \\
		&\textbf{p} \ge \textbf{0} \\
		& \textbf{p} \le k\textbf{e}_n  \tag{1*}
	\end{align*}

We now rescale this linear program. Define $\textbf{r} = \textbf{e}_n - \frac{1}{k} \textbf{p}$ and $s = n - \frac{1}{k}$. Substituting into \ref{eqn:GonzeLP2} and rearranging (noting that $\textbf{e}_n E(t) = 2\textbf{e}_{m(t)}$ and so on) we obtain the equivalent linear program $Prog_{B(t)}$ :
\begin{equation}
	\begin{aligned}[b] \label{eqn:vx_cover}
		\min_{\textbf{r},s} s \\
		\text{s.t.~~~} &\textbf{r}E(t) \ge \textbf{e}_{m(t)} \\
		&\textbf{e}_n\textbf{r}^{\mathsf{T}} = s \\
		&\textbf{r} \le \textbf{e}_n \\
		& \textbf{r} \ge \textbf{0}
	\end{aligned} 
\end{equation}

Let $s(t)$ be the minimum value attained by $Prog_{B(t)}$ and let $R_t$ be the set of optimal solutions $\textbf{r}$ to $Prog_{B(t)}$.  
If we restrict $\textbf{r}$ to having only entries in $\{0,1\}$ we get the standard integer linear program for the minimum vertex cover problem. A solution $\textbf{r}$ assigns weights $0$ or $1$ to the vertices of $G_t$ such that the total weight assigned is $s$. The condition $\textbf{r}E(t) \ge \textbf{e}_{m(t)}$ means that for each edge of $G_t$ the sum of the weights of the incident vertices is at least one --- the edge is `covered' by the vertices. The value of $s(t)$ is then the minimum total weight over all such vertex covers $\textbf{r}$.

Thus $Prog_{B(t)}$ is  the relaxation of the integer linear program for the minimum vertex cover problem. As a result $Prog_{A(t)}$ is in direct correspondence with the more well-studied fractional vertex cover problem.

It is well known (and straightforward to show) that $s(t)$ must always be half integral and, what is more, there is always an optimal solution $\textbf{r}$ with entries taken only from $\{0,\frac{1}{2},1\}$.
One immediate outcome of this correspondence is therefore that is an easy proof that the value of $k(t)$ must be of the form $\frac{2}{n + n_1}$ where $n_1$ is an integer between $0$ and $n$. This fact formed a crucial part of the proof in \cite{GJ16}.

Suppose we wanted to generalise this linear programming approach to get upper bounds on  the $k$-set rendezvous time for $k>3$. One can construct an analogous linear program to $Prog_{A(t)}$ and similarly transform it into an analogous linear program to $Prog_{B(t)}$. The resulting $Prog_{B(t)}$ is the fractional relaxation of a hypergraph minimum vertex cover problem where hyperedges of size $i$ must be covered $i-1$ times. That is, a solution $\textbf{r}$ assigns a weight to each vertex such that the sum of the weights of vertices in an $i$-edge is at least $i-1$. 

The hypergraph version does not have the nice half integral property that the graph version has, even in the smallest case $k=4$. 
The relevant fractional vertex cover problem is to find, for a hypergraph $H$ with edges of size $2$ and $3$, the minimum total weight $\vc(H)$ of a fractional vertex cover  that covers each $2$-edge with weight one and each $3$-edge with two.
The following Lemma proves that not only is this not half integral, but in fact there can be no integer $p$ where the set of possible values are of the form $i/p$ for $i \in \mathbb{N}$.

\begin{lemma} \label{lem:hyper_vx_cover}
For all $j \in \mathbb{N}$ there exists a hypergraph $H$ such that the minimum total weight $\vc(H)$  of a fractional vertex cover that covers each $2$-edge with weight one and each $3$-edge with weight two is $s/2^j$ for some odd integer $s$.
\end{lemma}
\begin{proof}
Let $H_j$ have $3j$ vertices labelled $x_1,\ldots,x_j$, $y_1,\ldots, y_j$ and $z_1,\ldots z_j$. Add the $2$-edges $\{x_1, y_1\}$, $\{y_1, z_1\}$ and $\{z_1,x_1\}$ and all $3$-edges of the form $\{x_i, y_{i-1}, z_i\}$, $\{y_i, z_{i-1}, x_i\}$, $\{z_i, x_{i-1}, y_i\}$ for $2 \le i \le j$. (You can also add $2$-edges that are subsets of a $3$-edge and it will not change the argument.)
Fix a  fractional vertex cover of $H_j$ with minimum total weight. For a vertex set $S$ let $w(S)$ denote the total weight of that set. 

Let $a_1,\ldots, a_j$ be positive rationals satisfying the $j$ simultaneous equations $2a_1 + a_2 = 2a_2 + a_3 = \ldots = 2a_{j-1} + a_j = 2a_j = 1$. Note that summing these equations gives $j = 2a_1 + 3\sum_{i=2}^{j}a_j$. We have:

\begin{align*}
\vc(H_j) = w(H_j) = ~ & \sum_{i=1}^{j} 2a_i(w(x_i)+w(y_i) +w(z_i))  + \sum_{i=2}^{j} a_i(w(x_{i-1})+w(y_{i-1})+ w(z_{i-1}))\\
= ~ & a_1\left( w(\{x_1, y_1\}) + w(\{y_1, z_1\}) + w(\{z_1,x_1\}) \right) \\
& + \sum_{i=2}^{j}a_i \left(w(\{x_i, y_{i-1}, z_i\}) + w(\{y_i, z_{i-1}, x_i\}) + w(\{z_i, x_{i-1}, y_i\}) \right) \\
\ge ~ &3a_1 +  \sum_{i=2}^{j}6a_i \\
= ~ & 2j - a_1
\end{align*}

We will show by induction that for $1 \le l \le j$, $a_{j+1- \ell} = s_{\ell}/2^{\ell}$ for some odd integer $s_{\ell}$. For the base case note $a_j = 1/2$. For the inductive step, let $1 \le \ell < j$ and suppose the statement is true for $a_{j+1-\ell}$. We have $a_{j-\ell} =  \frac{1 - a_{j+1-\ell}}{2}$ and so the inductive hypothesis holds. In particular, the value of $2j - a_1$ is $s/2^j$ for some odd integer $s$.

The final step of the proof is to show that this lower bound is attainable. Let $w(x_i) = w(y_i) = w(z_i) = 1 -  a_{j+1-i}$. It is easy to check that the total weight on the $2$-edges is one and the total weight on the $3$-edges is two. Furthermore, the total weight over all vertices is $3j - 3a_1 - \sum_{i=2}^j a_i = 2j - a_1$ as required.
\end{proof}
\begin{corollary}\label{cor:hyper_vx_cover}
For every $i,j \in \mathbb{N}$ there exists a hypergraph $H$ where $\vc(H)$ is $i/2^j$ more than an integer.
\end{corollary}
\begin{proof}
If $H_1 + H_2$ is the disjoint union of two hypergraphs $H_1$ and $H_2$, then $\vc(H_1+H_2) = \vc(H_1) + \vc(H_2)$. Thus by taking disjoint unions of the appropriate $H_j$ given by Lemma \ref{lem:hyper_vx_cover} we can obtain any residue of the form $i/2^j$ where $i,j \in \mathbb{N}$.
\end{proof}

The proof of the triple rendezvous time upper bound uses as a critical component that there are a small number of possible minimum values of the linear program $Prog_{A(t)}$, which follows from the half integrality of $Prog_{B(t)}$. We can see via the correspondence to the fractional vertex covers and Corollary \ref{cor:hyper_vx_cover} that there is no similarly straightforward bound on the number of possible minimum values when we generalise to $k>3$, and so this approach cannot be trivially generalised.

\section{Non-synchronizing Automata with Large Rendezvous Time}
\label{sec:constructions}

We now turn to the second half of the paper, which concerns rendezvous times in non-synchronizing automata.
We will prove a lower bound on $\rdvs^*(k,n)$ via a construction of a suitable automaton. To introduce the main idea of the construction we give the simpler $k=3$ case first.
\begin{theorem}\label{thm:construction k=3}
For every sufficiently large $n$, $\rdvs^*(3,n) > \frac{n^2}{8}$.
\end{theorem}
\begin{proof}
For every $n$ we will construct an automaton on $[n]$ where the minimal weight of a $k$-set is greater than $ \frac{n^2}{8}$. 

Partition $[n]$ into $A$ and $X$, where $|A| = \left\lfloor \frac{n}{4} \right\rfloor$.

Label the states of $A$ by $a_1,a_2,\ldots,a_{|A|}$ and label the states of $X$ by $x_1,x_2,\ldots,x_{|X|}$.

Take two functions $f$ and $g$ as follows, as shown in figure \ref{fig:example k=3} where $f$ is drawn in blue and $g$ in red.
\begin{align*}
f(x_t) &= x_{(t+1 \mod{|X|)}} \\
f(a_{|A|}) &= a_1 \\
f(a_j) &= x_j \text{ for } j \ne |A|\\
\\
g(x_t) &=   
	\begin{cases} 
   		x_{t+1} & \text{if } 1 \le t \le |A| - 1\\
   		x_{t - |A| + 1} & \text{if } t = |A|\\
   		x_t & \text{otherwise}
  	\end{cases} \\
g(a_j) &= a_{(j+1 \mod{|A|)}}
\end{align*}

\begin{figure}[h]
	\centering
	\includegraphics[scale=.6]{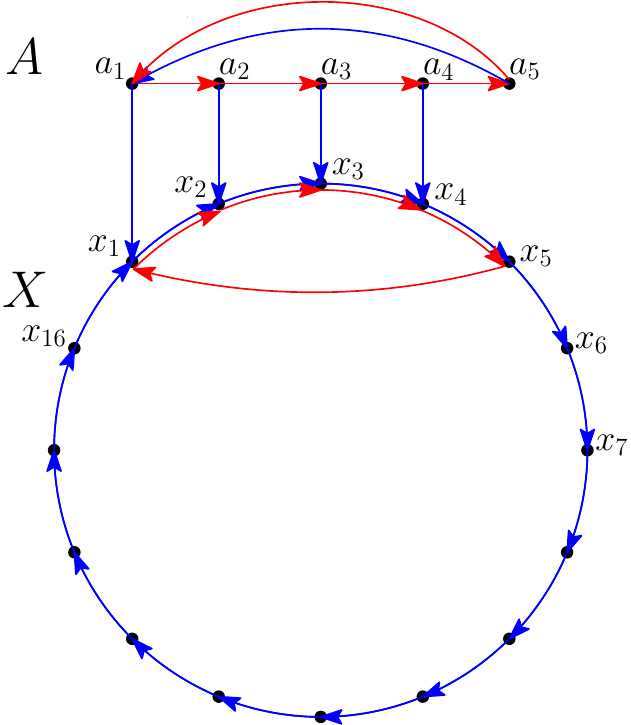}
	\caption{An example of the automaton used in the proof of Theorem \ref{thm:construction k=3} for $n = 21$,  where $f$ is drawn in blue and $g$ in red.}
	\label{fig:example k=3}
\end{figure}

Note that $f$ and $g$ restricted to $X$ are permutations on $X$ and so any set containing more than one state in $X$ cannot be synchronized. 
Moreover, any set containing three states in $A$ cannot be synchronized: the image of such a set under $g$ still has three states in $A$, and the image under $f$ contains two states in $X$. 

It follows that a synchronizable triple must contain two states in $A$ and one state in $X$.
Fix such a triple $S$ and consider a word that synchronizes this set acting on it. We will obtain that the triple of minimal weight is in fact $\{x_{|X|}, a_1, a_{|A|}\}$.

Note that for a shortest word from a triple to a singleton the first step must map a triple to a pair. In particular, the first map of the shortest word must be $f$, as $g$ is a permutation.
The triple $S$ must contain two states in $A$, one of which must be $a_{|A|}$ otherwise applying $f$ gives two states in $X$. Let the other be $a_t$, where $1 \le t \le |A|-1$. 
After applying $f$, we have the states $a_1$ and $x_t$, which must be the only state in $X$. 

Note that 
$$
fg^{l-1}\left(a_1\right) = 
\begin{cases} 
	a_1 & \text{if } l \equiv 0 \pmod{|A|}\\
	x_{(l \mod{|A|})} & \text{otherwise}
\end{cases}
$$
and for  $1 \le t \le |A|-1$,
$$fg^{l-1}\left(x_t\right) =
\begin{cases} 
	x_{|A| + 1} & \text{if } t + l - 1 \equiv 0 \pmod{|A|} \\
	x_{(t+l \mod{|A|})} & \text{otherwise.}
\end{cases}
.$$

This means that applying $fg^{l-1}$ gives two states in $X$ for every $l \not\equiv 0 \pmod{|A|}$. Thus the next step must be to apply  $fg^{l-1}$ where $l$ is some multiple of $|A|$. This sends $a_1$ and $x_t$ to themselves unless $t=1$, in which case $x_t$ is sent to $x_{|A|+1}$.

To further reduce the size of the set, we must map $x_{|A|+1}$ and $a_1$ to the same state. To do this, we must move the state in position $x_{|A| + 1}$ round through $x_{|A| +2}, x_{|A| +3},\ldots$ until we reach $x_{|X|}$, without moving the second state that is currently in $A$ into $X$ as we do so.

Suppose we have just applied $f$, and we now want to move $x_s$ to $x_{s+1}$ without adding any extra states into $X$ (where $s$ is some value not in $\{|X|, 1,2,3 \ldots, |A|\}$). 
Since we have just applied $f$, the state in $A$ must be at position $a_1$ (having just come from position $a_{|A|}$).
We need to apply $f$ to move $x_s$, but we can only apply $f$ when the state in $A$ is at position $a_{|A|}$ and so we must first apply $g^{|A|-1}$ to move the state at $a_1$ to be at $a_{|A|}$. Only then can we apply $f$, and so the shortest word moving $x_s$ to $x_{s+1}$ is $fg^{|A|-1}$.

Repeatedly applying this, we have that the shortest word squashing a triple to a singleton is $f\left(fg^{|A|-1}\right)^{(|X| - (|A|+1))}fg^{|A|-1}f$ which has length 
$$1 + (|X| - |A|)|A| +1 =  \left(n - 2\left\lfloor \frac{n}{4} \right\rfloor \right) \left\lfloor \frac{n}{4}\right\rfloor + 2 >  \frac{n^2}{8}.$$
\end{proof}

The general case extends the construction given in Theorem \ref{thm:construction k=3}. We still have two mappings and a set of states $X$ on which both mappings act as permutations, meaning that any synchronizable set has at most one state in $X$. Rather than having a single gadget $A$  we will need $k-2$ gadgets $A_0,A_1,A_{k-3}$, each with the same structure as $A$ but of coprime sizes. 

To synchronize a $k$-set we will need to apply a mapping $f$ to move a state around $X$. As before, we will not be able to apply $f$ without first applying the other mapping $g$ several times to move the state in each $A_i$ from $a^{(i)}_1$ to $a^{(i)}_0$. Because we chose the $A_i$ to have coprime sizes, each such move will neccessitate many applications of $g$.

\begin{theorem}\label{thm:construction}
Let $k\ge 3$. For every $n$ sufficiently large, $\rdvs^*(k,n) \ge \frac{4}{3}\left(\frac{n}{4k}\right)^{k-1} $.
\end{theorem}

\begin{proof}
Fix the integer $k$. 
For every $n$ we will construct an automaton on $[n]$ where the minimal weight of a $k$-set is $\frac{4}{3}\left(\frac{n}{4k}\right)^{k-1}$. 

Partition $[n]$ into $A_0,A_1,A_2,\ldots,A_{k-3}$ and $X$, where ${\frac{n}{4k} \le |A_i| \le \frac{n}{3k}}$ and \\
$\gcd\{A_0 ,A_1 , A_2 , \ldots, A_{k-3} \} = 1$. This is possible for $n$ sufficiently large, for example by the prime number theorem. 

Label the states in each $A_i$ by $a^{(i)}_1,a^{(i)}_2,a^{(i)}_3,\ldots$ and label the states of $X$ by $x_1,x_2,x_3\ldots$.
Let $q = \left\lfloor\frac{2n}{3k}\right\rfloor$.

Take two functions $f$ and $g$ as follows, as shown in figure \ref{fig:example} where $f$ is drawn in blue and $g$ in red.
\begin{align*}
f(x_t) &= x_{(t+1 \mod{|X|})} \\
f\left(a^{(i)}_j\right) &= 
	\begin{cases} 
   		a^{(i)}_1 & \text{if } j=|A_i| \\
   		x_{iq + j} & \text{otherwise}
  	\end{cases} \\
\\
g(x_t) &=   
	\begin{cases} 
   		x_{t+1} & \text{if } iq + 1 \le t \le iq + |A_i| - 1 \text{ for some } i\\
   		x_{t - |A_i| + 1} & \text{if } t = iq + |A_i|  \text{ for some } i\\
   		x_t & \text{otherwise}
  	\end{cases} \\
g(a^{(i)}_j) &= a^{(i)}_{(j+1 \mod{|A_i|})}
\end{align*}

\begin{figure}
	\centering
	\includegraphics[scale=.6]{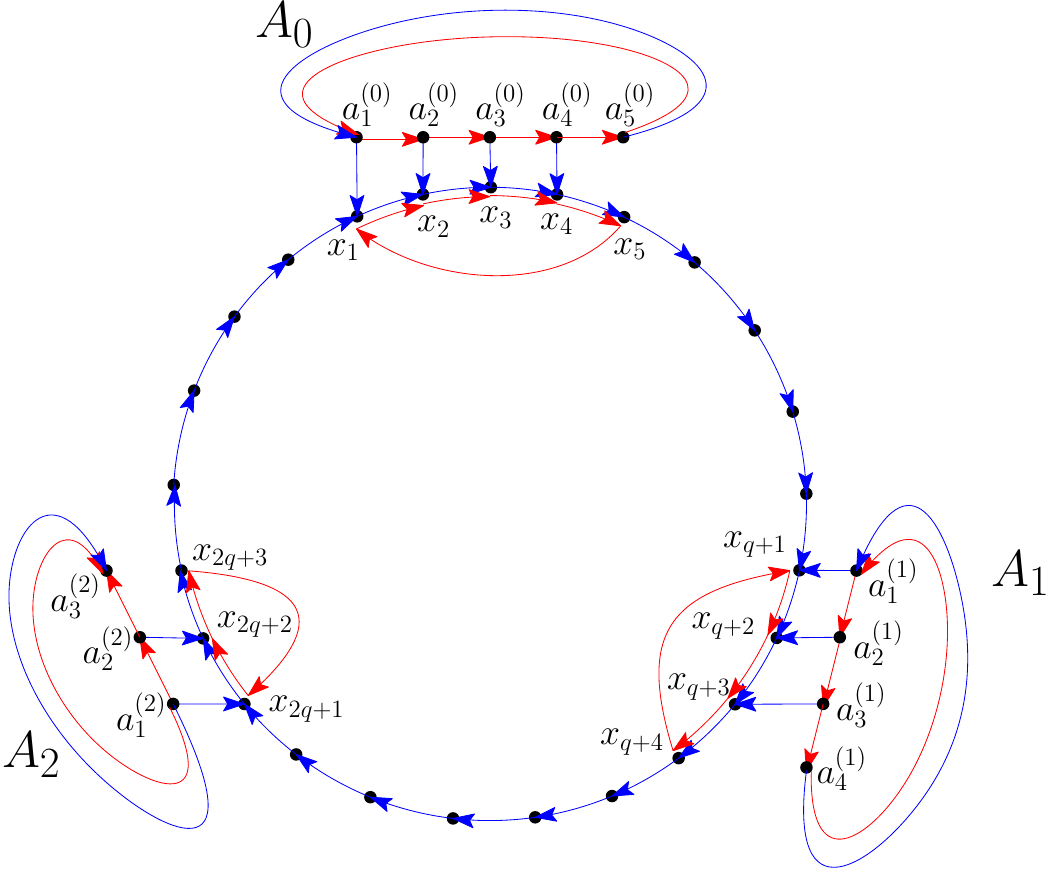}
	\caption{An example of the automaton used in the proof of Theorem \ref{thm:construction}  for $k=5$, where $f$ is drawn in blue and $g$ in red.}
	\label{fig:example}
\end{figure}

Note that $f$ and $g$ restricted to $X$ are permutations on $X$ and so any set containing more than one state in $X$ cannot be synchronized. 

Moreover, any set containing three states in some $A_i$ cannot be synchronized: the image of such a set under $g$ still has three states in $A_i$, and the image under $f$ contains two states in $X$. 
Similarly, any set containing two states in $A_i$ and two states in $A_j$ for some distinct $i$ and $j$ also cannot be synchronized.

It follows that a synchronizable set of size $k$ must contain two states in some $A_i$, one state in every other $A_j$ and one state in $X$.
Fix such a set $S$ and consider a word that synchronizes this set acting on it. 

For a shortest word from a triple to a singleton the first step must map a triple to a pair and so the first map must be $f$.
The set $S$ contains two states in $A_i$, one of which must be $a^{(i)}_{|A_i|}$ else applying $f$ gives two states in $X$. Let the other be $a^{(i)}_t$, where $1 \le t \le |A_i|-1$. 
After applying $f$, we have the states $a^{(i)}_1$ and $x_{iq+t}$, which must be the only state in $X$. 

Note that 
$$
fg^{l-1}\left(a^{(i)}_1\right) = 
\begin{cases} 
	a^{(i)}_1 & \text{if } l \equiv 0 \pmod{|A_i|}\\
	x_{iq+(l \mod{|A_i|})} & \text{otherwise}
\end{cases}
$$
and 
$$fg^{l-1}\left(x_{iq+t}\right) =
\begin{cases} 
	x_{iq+|A_i| + 1} & \text{if } t + l - 1 \equiv 0 \pmod{|A_i|} \\
	x_{iq+(t+l \mod{|A_i|})} & \text{otherwise}
\end{cases}
.$$
Since $1 \le t \le |A_i|-1$ this means that applying $fg^{l-1}$ gives two states in $X$ for any $l \not\equiv 0 \pmod{|A_i|}$. Thus the next step must be to apply  $fg^{l-1}$ where $l$ is some multiple of $|A_i|$. This sends $a^{(i)}_1$ and $x_{iq+t}$ to themselves unless $t=1$, in which case $x_{iq+t}$ is sent to $x_{iq+|A_i| + 1}$.

To further reduce the size of the set, we must map the state in $X$ and some state in some $A_j$ to the same state. To do this, we must move the state $x_{iq+|A_i| + 1}$ in $X$ round to be in $\{ x_{jq},  x_{jq+1}, \ldots, x_{jq+|A_j| - 2}\}$, without adding extra states to $X$ as we do so.

Suppose we have just applied $f$, and we now want to move $x_s$ to $x_{s+1}$ without adding any extra states into $X$ (where $s$ is some value not in $\{ x_{jq+1},  x_{jq+2}, \ldots, x_{jq+|A_j| - 1}\}$ for any $j$). 
Since we have just applied $f$, the state in each $A_j$ must be at position $a^{(j)}_1$ (having just come from position $a^{(j)}_{|A_j|}$).
We must apply $f$ to move $x_s$, but we can only apply $f$ when for each $A_j$, the state in $A_j$ is at position $a^{(j)}_{|A_j|}$. 
Thus we must use $g$ to move the state at $a^{(j)}_1$ to be at $a^{(j)}_{|A_j|}$ for each $j$. 

The number of times $g$ is applied must be congruent to $-1$ modulo $|A_j|$ for all $j$. 
Since $|A_0|, |A_1|, \ldots, |A_{k-3}|$ are coprime, the smallest such number is $\prod_{j=0}^{k-3} |A_j| - 1$. 
This is followed by an application of $f$ and so it takes at least $\prod_{j=0}^{k-3} |A_j|$ steps to move $x_s$ to $x_{s+1}$.

Applying this repeatedly, we see that the length of a word taking the state in $X$ from $x_{iq+|A_i| + 1}$ to some state of the form $\{ x_{jq},  x_{jq+1}, \ldots, x_{jq+|A_j| - 2}\}$ without introducing a second state to $X$ must be at least
$$
(q - (|A_i| - 1))\prod_{j=0}^{k-3}  |A_j| 
\ge \left(\frac{2n}{3k} - \frac{n}{3k}\right)\left(\frac{n}{4k}\right)^{k-2} = \frac{4}{3}\left(\frac{n}{4k}\right)^{k-1}.
$$
 
\end{proof}

Theorem \ref{thm:construction}, together with the observation that a minimal length path from some $k$-set to a singleton passes through each set of size $<k$ at most once, tells us that $\rdvs^*(k,n) = \Theta\left(n^{k-1}\right)$ for fixed $k$. Since $\Rdvs^*(k,n) \ge \rdvs^*(k+1,n) - 1$ we have as an immediate consequence that $\Rdvs^*(k,n) = \Theta\left(n^{k-1}\right)$.

These results are very different from the situation for synchronizing automata.
One thing we learn therefore is that any bound on the $k$-set rendezvous time $\rdvs(k,n)$ must use the fact that the automata are synchronizing as a crucial part. In particular, this impacts any attempt at a proof or improved bound for \cernys conjecture that relies on bounding the $k$-set rendezvous time --- such a proof must use somewhere that all pairs (and all sets) are synchronizable.

\section{Open Questions}

As mentioned at the end of Section \ref{sec:sync}, it may be possible that the tools used to prove Theorems \ref{thm:k=3}, \ref{thm:k=4} and \ref{thm:k=5} could be extended further and combined with new ideas to give improved upper bounds on  $\rdvs(k,n)$ for $k>3$. To do so, one would have to strengthen Lemma \ref{lem:1} and/or Lemma \ref{lem:2}.

We believe something stronger may be true, at least for small $k$. We know from \cite{GJ16} that the triple rendezvous time $\rdvs(3,n) \lessapprox 0.1545 n^2 +O(n)$. However, the best known lower bound to $\rdvs(3,n)$ is $n+3$. Given the lack of any examples to the contrary, we conjecture that the triple rendezvous time is in fact linear in $n$.
\begin{conjecture} \label{conj:triple rdv} There exists some constant $c$ such that $\rdvs(3,n) \le c n$ for all $n$.
\end{conjecture}
Any techniques involved in the proof of Conjecture \ref{conj:triple rdv} may well generalise to give improved bounds on the $k$-set rendezvous time $\rdvs(k,n)$ and potentially $\rdvs(n,n)$, the \cerny bound itself.

We can also ask about improved bounds on $\Rdvs(3,n)$, which is known to be between $\frac{2}{3}n^2 + O(n)$ and $n^2 + O(n)$. 
\begin{question} Is there some constant $c < 1$ such that $\Rdvs(3,n) \le cn^2 + O(n)$? 
\end{question}
A positive answer to  this question would give an improvement to the Frankl--Pin bound for the length of a shortest reset word from  $\frac{n^3}{6}$ to  $c'n^3$ for come constant $c'<\frac{1}{6}$. The reason for this is that for $k \ge 3$ we would have from any $k$-set there is a path to a  $(k-2)$-set of length $\le cn^2$. This is an improvement on the Frankl--Pin bound $\binom{n-k+2}{2} + \binom{n-k+3}{2}$ for $k < (1-\sqrt{c})n$, that is, a linear proportion of all $k$.

There is nothing special about triples here: an improved upper bound on $\Rdvs(k,n)$ for any fixed $k$ would give a improvement on the Frankl--Pin bound in a similar way. We also don't have to be restricted to paths from $k$-sets to singletons --- one can ask the same questions about the shortest path from a $k$-set to an $l$-set for any $k > l$ and draw similar conclusions from any improved bounds.

Theorem \ref{thm:construction} shows that for fixed $k$ we have $\rdvs^*(k,n) = \Theta\left(n^{k-1}\right)$. A natural question to ask is what are the correct asymptotics for $\rdvs^*(k,n)$? In the case $k=3$ we have $\frac{n^2}{8} \le \rdvs^*(k,n) \le \frac{n^2 - n -1}{2}$.
\begin{question}
Is there an automaton which attains $\rdvs^*(3,n) = (\frac{1}{2}+o(1))n^2$?
\end{question}
An upper bound on the minimum weight of a triple $\rdvs^*(3,n)$ is the total number of synchronizable pairs plus one. To get  a minimum weight triple of weight $(\frac{1}{2}+o(1))n^2$ we would need the automaton to be almost synchronizing in the sense that all but an arbitrarily small proportion of pairs are synchronizable.

Consider the construction given in the proof of Theorem \ref{thm:construction k=3}. We know that a pair of states both in $X$ is not synchronizable. In fact, it is straightforward to check that only pairs of the following forms are synchronizable:
\begin{itemize}
\item $\{a_i,x_s\}$ for $i \in \{1,2,3,\ldots, |A|\}$ and $s \not \in \{1,2,3,\ldots,|A|\}$,
\item $\{a_i,x_i\}$ for $i \in \{1,2,3,\ldots, |A|\}$,
\item $\{a_1, x_{|A|}\} $ and $ (a_i,x_{i-1})$ for $i \in \{2,3,\ldots, |A|\}$,  and 
\item $\{a_i,a_{(i+1 \mod{|A|})} \}$ for $i \in \{1,2,3,\ldots, |A|\}$.
\end{itemize}  
In particular, the automaton has $|A|\left(|X| - |A|\right) + 3|A| = \frac{n^2}{8} + O(n)$ synchronizable pairs. We have that the number of synchronizable pairs and the minimum weight of a triple are asymptotically equal in this example. Is it possible to construct an automaton with this same property where a larger proportion of pairs are synchronizable?

We also note that our construction is not strongly connected. Could there be a strongly connected automaton with the same properties, that is, with only two mappings and needing time $\Omega(n^k)$ to synchronize a $k$-set?

\end{document}